\documentclass[a4paper,11pt]{article}

\usepackage[textwidth=125mm, textheight=195mm]{geometry}
\usepackage[english]{babel}
\usepackage{graphicx}
\usepackage{amsmath}
\usepackage{amsfonts}
\usepackage{amsthm}
\usepackage{epstopdf}

\begin{document}

\newcommand{\wk}{\mbox{$\,<$\hspace{-5pt}\footnotesize )$\,$}}

\numberwithin{equation}{section}
\newtheorem{teo}{Theorem}
\newtheorem{lemma}{Lemma}

\newtheorem{coro}{Corollary}
\newtheorem{prop}{Proposition}
\theoremstyle{definition}
\newtheorem{definition}{Definition}
\theoremstyle{remark}
\newtheorem{remark}{Remark}

\newtheorem{scho}{Scholium}
\newtheorem{open}{Question}
\newtheorem{example}{Example}
\numberwithin{example}{section}
\numberwithin{lemma}{section}
\numberwithin{prop}{section}
\numberwithin{teo}{section}
\numberwithin{definition}{section}
\numberwithin{coro}{section}
\numberwithin{figure}{section}
\numberwithin{remark}{section}
\numberwithin{scho}{section}

\bibliographystyle{abbrv}

\title{On Legendre curves in normed planes}
\date{}

\author{Vitor Balestro\footnote{Corresponding author} \footnote{The first named author would like to thank Prof. Marcos Craizer, who brought to his attention the study of differential geometry from the point of view of singularity theory.} \\ CEFET/RJ Campus Nova Friburgo \\ 28635000 Nova Friburgo \\ Brazil \\ vitorbalestro@id.uff.br \and Horst Martini \\ Fakult\"{a}t f\"{u}r Mathematik \\ Technische Universit\"{a}t Chemnitz \\ 09107 Chemnitz\\ Germany \\ martini@mathematik.tu-chemnitz.de \and  Ralph Teixeira \\ Instituto de Matem\'{a}tica e Estat\'{i}stica  \\ Universidade Federal Fluminense \\24210201 Niter\'{o}i\\ Brazil \\ ralph@mat.uff.br}

\maketitle

\begin{abstract} Legendre curves are smooth plane curves which may have singular points, but still have a well defined smooth normal (and corresponding tangent) vector field. Because of the existence of singular points, the usual curvature concept for regular curves cannot be straightforwardly extended to these curves. However, Fukunaga, and Takahashi defined and studied functions that play the role of curvature functions of a Legendre curve, and whose ratio extend the curvature notion in the usual sense. Going into the same direction, our paper is devoted to the extension of the concept of circular curvature from regular to Legendre curves, but additionally referring not only to the Euclidean plane. For the first time we will extend the concept of Legendre curves to normed planes. Generalizing in such a way the results of the mentioned authors, we define new functions that play the role of circular curvature of Legendre curves, and tackle questions concerning existence, uniqueness, and invariance under isometries for them. Using these functions, we study evolutes, involutes, and pedal curves of Legendre curves for normed planes, and the notion of contact between such curves is correspondingly extended, too. We also provide new ways to calculate the Maslov index of a front in terms of our new curvature functions. It becomes clear that an inner product is not necessary in developing the theory of Legendre curves. More precisely, only a fixed norm and the associated orthogonality (of Birkhoff type) are necessary.

\end{abstract}


\bigskip


\section{Introduction}

The concept of \emph{curvature} of regular curves in the Euclidean plane can be extended to normed planes in several ways (see \cite{Ba-Ma-Sho} for an exposition of the topic,  and \cite{martiniandwu} refers, more generally, to classical curve theory in such planes). One of the curvature types obtained by these extensions, namely the \emph{circular curvature}, can be regarded as the inverse of the radius of a 2nd-order contact circle at the respective point of the curve. Therefore it turns out that the investigation of the differential geometry of these curves from the viewpoint of singularity theory is also due to this context (see \cite{izumiya} and \cite[Section 9]{Ba-Ma-Sho}). In the Euclidean subcase, the concept of curvature can be carried over to certain curves containing singular points. This was done by Fukunaga and Takahashi in \cite{Fu-Ta}, and the aim of the present paper is to investigate this framework more generally for normed planes, using the concept of circular curvature, and also extending the usual inner product orthogonality to \emph{Birkhoff orthogonality}. \\

We start with some basic definitions. A \emph{normed} (or \emph{Minkowski}) \emph{plane} $(X,\|\cdot\|)$ is a two-dimensional real vector space $X$ endowed with a norm $||\cdot|| :X\rightarrow\mathbb{R}$, whose \emph{unit ball}  is the set $B:=\{x \in X: ||x|| \leq 1\}$, namely a compact convex set centered at the origin $o$ which is an interior point of $B$. The boundary $S:=\{x\in X:||x|| = 1\}$ of $B$ is called the \emph{unit circle}, and all homothetical copies of $B$ and $S$ will be called \emph{Minkowski balls} and \emph{Minkowski circles}, respectively. We will always assume that the plane is \emph{smooth}, which means that $S$ is a smooth curve, and also \emph{strictly convex}, meaning that $S$ does not contain straight line segments. In a normed plane $(X,||\cdot||)$ we define an orthogonality relation by stating that two vectors $x,y \in X$ are \emph{Birkhoff orthogonal} (denoted by $x \dashv_B y$) whenever $||x+ty|| \geq ||x||$ for each $t \in \mathbb{R}$. Geometrically this means that if $x \dashv_By$ and $x \neq 0$, then the Minkowski circle centered at the origin which passes through $x$ is supported by a line in the direction of $y$. Useful references with respect to Minkowski geometry (i.e., the geometry of finite dimensional real Banach spaces) are \cite{thompson}, \cite{martini1}, and \cite{martini2}; for orthogonality types in Minkowski spaces we refer to \cite{alonso}.\\

One should notice that Birkhoff orthogonality is not necessarily a symmetric relation. Actually, we may endow the plane with a new associated norm which reverses the orthogonality relation. To do so, we fix a nondegenerate symplectic bilinear form $[\cdot,\cdot]:X\times X\rightarrow \mathbb{R}$ (which is unique up to rescaling) and define the associated \emph{anti-norm} to be
\begin{align*} ||x||_a = \sup\{|[x,y]|:y\in S\}, \ \ x \in X.
\end{align*}
It is easily seen that $||\cdot||_a$ is a norm on $X$, and that it reverses Birkhoff orthogonality. Moreover, the \emph{unit anti-circle} (i.e., the unit circle of the anti-norm) solves the isoperimetric problem in the original Minkowski plane (see \cite{Bus6}). The planes where Birkhoff orthogonality is symmetric are called \emph{Radon planes}, and their unit circles are called \emph{Radon curves}. In this case, we clearly have that the unit circle and the unit anti-circle are homothets, and we will always assume that the fixed symplectic bilinear form is rescaled in such way that they coincide. A comprehensive exposition on this topic is \cite{martiniantinorms}.\\

A smooth curve $\gamma:J\rightarrow X$ is said to be \emph{regular} if $\gamma'(t) \neq 0$ for every $t \in J$. If a curve is not regular, then a point, where the derivative vanishes, is called a \emph{singular point} of $\gamma$. The \emph{length} of a curve $\gamma:[a,b]\rightarrow X$ is defined as usual in terms of the norm by
\begin{align*} l(\gamma) := \sup_P\sum_{j=1}^n||\gamma(t_{j})-\gamma(t_{j-1})||,
\end{align*}
where the supremum is taken over all partitions of $P = \{a = t_0,...,t_n = b\}$ of $[a,b]$.  It is clear that we can define here the standard arc-length parametrization, and that if $s$ is an arc-length parameter in $\gamma$, then $||\gamma'(s)|| = 1$. We head now to define the \emph{circular curvature} for a regular curve $\gamma:[0,l(\gamma)]\rightarrow X$ parametrized by arc-length (for the sake of simplicity). To do so, let $\varphi(u):[0,l(S)]\rightarrow X$ be a parametrization of the unit circle by arc-length. Let $u(s):[0,l(\gamma)]\rightarrow[0,l(S)]$ be the function such that $\gamma'(s) = \frac{d\varphi}{du}(u(s))$. Then the \emph{circular curvature} of $\gamma$ at $\gamma(s)$ is defined as
\begin{align*} k(s) := u'(s).
\end{align*}
We define the \emph{left normal field} of $\gamma$ to be the unit vector field $\eta:[0,l(\gamma)]\rightarrow S$ such that $\eta(s) \dashv_B \gamma'(s)$ and $[\eta(s),\gamma'(s)] > 0$ for each $s \in [0,l(\gamma)]$. Writing $\gamma'(s) = \frac{d\varphi}{du}(u(s))$, we have that the left normal field is given by $\eta(s) = \varphi(u(s))$. Therefore, we have the Frenet-type formula
\begin{align*} \eta'(s) = u'(s)\frac{d\varphi}{du}(u(s)) = k(s)\gamma'(s).
\end{align*}
The \emph{center of curvature} of $\gamma$ at $\gamma(s)$ is the point $c(s) := \gamma(s) - k(s)^{-1}\eta(s)$, and we call the number $\rho(s) :=k(s)^{-1}$ the \emph{curvature radius} of $\gamma$ at $\gamma(s)$. The circle centered in $c(s)$ and having radius $\rho(s)$ is the \emph{osculating circle} of $\gamma$ at $\gamma(s)$. It is easily seen that this circle has 2nd-order contact with $\gamma$ at $\gamma(s)$. From the viewpoint of singularity theory, the \emph{distance squared function} of $\gamma$ to a point $p \in X$ is the function $D_p(s):=||\gamma(s)-p||^2$. We can obtain the centers of curvature of a given curve as follows.

\begin{prop} Let $\gamma:[0,l(\gamma)]\rightarrow X$ be a smooth and regular curve parametrized by arc length. Then the function $D_p(s)=||\gamma(s)-p||^2$ is such that $D_p'(s_0) = D_p''(s_0) = 0$ if and only if $p$ is the center of curvature of $\gamma$ at $\gamma(s_0)$.
\end{prop}
\begin{proof} See \cite[Proposition 9.1]{Ba-Ma-Sho}.

\end{proof}

Throughout the text, we will call the circular curvature simply \emph{curvature}, and the left normal field will be referred to as \emph{normal field}.

\section{Curvature of curves with singularities} \label{curvature}

The main objective of this paper is to extend and study the concept of curvature for curves in normed planes which have certain types of ``well-behaving'' singularities. Roughly speaking, in certain situations a curve can have a singularity, but we are still able to derive a natural tangent direction corresponding with the respective curve point. For example, let $\gamma(t):I\rightarrow X$ be a curve, and assume that $\gamma$ has a (unique, for the sake of simplicity) isolated singularity at $t_0 \in I$ (i.e., $\gamma'(t)$ does not vanish in a punctured neighborhood of $t_0$). If both limits
\begin{align*} \lim_{t\rightarrow t_0^{\pm}}\frac{\gamma'(t)}{||\gamma'(t)||}
\end{align*}
exist and are equal up to the sign, then we can naturally define a field of tangent (or normal) directions through the entire $\gamma$. This kind of singularity appears, for example, in \emph{evolutes} of regular curves (see \cite[Section 9]{Ba-Ma-Sho}).\\

In singularity theory, submanifolds with singularities but well-defined tangent spaces are usually called \emph{frontals} (cf. \cite{ishikawa1}). If the ambient space is two-dimensional, then these submanifolds are precisely the curves which have well-defined tangent field, even if they contain singularities. Such curves were studied by Fukunaga and Takahashi in \cite{Fu-Ta}, \cite{Fu-Ta2}, \cite{Fu-Ta3}, and \cite{Fu-Ta4}. Heuristically speaking, the existence of a well-defined tangent field has no relation to the metric of the plane. Therefore, we can re-obtain the definitions posed by the mentioned authors, but now regarding the usual tools and machinery of planar Minkowski geometry. \\

We define a \emph{Legendre curve} to be a smooth map $(\gamma,\eta):I\rightarrow X\times S$ such that $\eta(t) \dashv_B \gamma'(t)$ for every $t \in I$. If a Legendre curve is an \emph{immersion} (i.e., if the derivatives of $\gamma$ and $\eta$ do not vanish at the same time), then we call it a \emph{Legendre immersion}. A curve $\gamma:I\rightarrow X$ is said to be a \emph{frontal} if there exists a smooth map $\eta:I\rightarrow S$ such that $(\gamma,\eta)$ is a Legendre curve. Finally, we say that $\gamma$ is a \emph{front} if there exists a smooth map $\eta:I\rightarrow S$ such that $(\gamma,\eta)$ is a Legendre immersion. \\

Since we are dealing with smooth and strictly convex normed planes, it follows that Birkhoff orthogonality is unique on both sides. Define the map $b:X\setminus\{o\}\rightarrow S$ (where $o$ again denotes the origin of the plane) which associates to each $v \in X\setminus\{o\}$ the unique vector $b(v) \in S$ such that $v \dashv_B b(v)$ and $[v,b(v)] > 0$. A Legendre curve is defined heuristically by guaranteeing the existence of a \emph{normal field} to $\gamma$, instead of a tangent field. But now we simply use the map $b$ to define a ``tangent field''. We just have to define, for a Legendre curve $(\gamma,\eta):I\rightarrow X\times S$, the vector field $\xi(t):=b(\eta(t))$. Of course, $\xi(t)$ points in the direction of $\gamma'(t)$. Then there exists a smooth function $\alpha:I\rightarrow\mathbb{R}$ such that
\begin{align}\label{1stcurvature} \gamma'(t) = \alpha(t)\xi(t), \ \ t \in I.
\end{align}
Also, since $\eta'(t)$ supports the unit circle at $\eta(t)$, it follows that there exists a smooth function $\kappa:I\rightarrow\mathbb{R}$ such that
\begin{align}\label{2ndcurvature} \eta'(t) = \kappa(t)\xi(t), \ \ t \in I.
\end{align}

We call the pair $(\alpha,\kappa)$ the \emph{curvature} of the Legendre curve $(\gamma,\eta)$ with respect to the parameter $t$. 
This terminology makes sense since it is easy to see that the curvature of a Legendre curve depends on its parametrization. To justify why this pair of functions represents an a\-na\-lo\-gous concept of curvature for Legendre curves, we will show that it yields the usual (circular) curvature of a regular curve.

\begin{lemma}\label{relcurv} Let $\gamma:I\rightarrow X$ be a regular curve in a normed plane. Clearly, if $\eta:I\rightarrow S$ is its normal vector field, then $(\gamma,\eta)$ is a Legendre curve. Therefore, its circular curvature $k:I\rightarrow\mathbb{R}$ is given by
\begin{align*} k(t) = \frac{\kappa(t)}{\alpha(t)},
\end{align*}
where $\kappa$ and $\alpha$ are defined as above.
\end{lemma}
\begin{proof} For the Legendre curve $(\gamma,\eta)$ we have the equalities (\ref{1stcurvature}) and (\ref{2ndcurvature}). Notice that, since $\gamma$ is regular, the function $\alpha$ does not vanish. Hence we may write
\begin{align*} \eta'(t) = \frac{\kappa(t)}{\alpha(t)}\gamma'(t).
\end{align*}
On the other hand, let $s$ be an arc-length parameter in $\gamma$ and, as usual, let $\varphi(u)$ be an arc-length parametrization of the unit circle. We denote the derivative with respect to $s$ by a superscribed dot, and write
\begin{align*} \dot{t}(s)\gamma'(t) = \dot{\gamma}(s) = \frac{d\varphi}{du}(u(s)),
\end{align*}
where $u(s)$ is as in the definition of circular curvature. We have that $k(s) = \dot{u}(s)$ and $\eta(s) = \varphi(u(s))$. Differentiating this last equality, we get
\begin{align*} \dot{t}(s)\eta'(t) = \dot{u}(s)\frac{d\varphi}{du}(u(s)) = \dot{u}(s)\dot{\gamma}(s) = \dot{u}(s)\dot{t}(s)\gamma'(t),
\end{align*}
and since $\dot{t}(s)$ does not vanish, it follows that $\eta'(t) = \dot{u}(s)\gamma'(t) = k(t)\gamma'(t)$. This gives the desired equality.

\end{proof}

\begin{remark}\label{diffb} When working in the Euclidean plane, one gets a second Frenet-type formula by differentiating the field $\xi(t)$, and the same curvature function $\kappa(t)$ is obtained (see \cite{Fu-Ta}). This is not the case here. The first problem that appears is that the derivative of $\xi(t)$ does not necessarily point in the direction of $\eta(t)$. We can overcome this problem by restricting ourselves to Radon planes. However, even in this small class of norms we do not re-obtain the same curvature function. Indeed, since $\xi(t) = b(\eta(t))$, we have
\begin{align*} \xi'(t) = \mathrm{D}b_{\eta(t)}(\eta'(t)) = \kappa(t)\mathrm{D}b_{\eta(t)}(\xi(t)),
\end{align*}
where $\mathrm{D}b$ denotes the usual differential of the map $b:X\setminus\{o\}\rightarrow S$, which is no longer a (linear) rotation. It turns out that, since the considered plane is Radon, the vector $\mathrm{D}b_{\eta(t)}(\xi(t))$ is a positive multiple of the vector $-\eta(t)$, but it is not necessarily unit. If we define the map $\rho:S\rightarrow \mathbb{R}$ by $\rho(v) = ||\mathrm{D}b_{v}(b(v))||$, then we may write
\begin{align}\label{3rdcurvature} \xi'(t) = -\kappa(t)\rho(\eta(t))\eta(t).
\end{align}
The function $\rho$ is constant, however, if and only if the plane is Euclidean (see \cite{vitoremad} for a proof). If we return to the general case, we clearly have
\begin{align}\label{xideriv} \xi'(t) = -\kappa(t)\rho(\eta(t))b(\xi(t)),
\end{align}
and this equality will be used in Section \ref{evoinvo}, where the function $\rho$ will appear in the curvature pair of the evolute of a front. 
\end{remark}

At this point, we have seen that we can extend the definitions, which are common for the Euclidean subcase, in a way that everything still makes sense and has analogous behavior. However, a simple question arises: if a certain fixed curve in the plane is a Legendre curve (immersion) with respect to a fixed norm, is it then necessarily a Legendre curve (immersion) with respect to any other (smooth and strictly convex) norm? The answer is positive, and we can briefly explain the argument. Let $S_1$ and $S_2$ be unit circles with respect to two different norms, and denote the respective Birkhoff orthogonality relatons by $\dashv_B^1$ and $\dashv_B^2$. Consider the map $T:S_1\rightarrow S_2$ which associates each $v \in S^1$ to the unique $T(v) \in S_2$ such that $T(v) \dashv_B^2 b_1(v)$ and $[T(v),b_1(v)] > 0$, where $b_1$ is the usual map $b$ of the geometry given by $S_1$. The map $T$ is clearly smooth, and if $(\gamma,\eta)$ is a Legendre curve (immersion) with respect to the norm of $S_1$, then $(\gamma,T(\eta))$ is a Legendre curve (immersion) with respect to the norm of $S_2$. The details are left to the reader. 

\section{Existence, uniqueness, and invariance under isometries}

This section is concerned with natural questions regarding the generalized objects that we have defined. We start by asking whether or not there exists a corresponding Legendre curve whose curvature is given by certain fixed smooth functions $\kappa,\alpha:I\rightarrow\mathbb{R}$. For simplicity, throughout this section we assume that $I = [0,c]$.

\begin{teo}[Existence theorem] Let $(\alpha,\kappa):I\rightarrow\mathbb{R}^2$ be a smooth function. Then there exists a Legendre curve $(\gamma,\eta):I\rightarrow X\times S$ whose curvature is $(\alpha,\kappa)$.
\end{teo}
\begin{proof} First, define the function $u:I\rightarrow\mathbb{R}$ by
\begin{align*} u(t) = \int_0^t\kappa(s)ds, \ \ t \in I.
\end{align*}
Now, define $\eta:I\rightarrow S$ by $\eta(t) = \varphi(u(t))$, and $\gamma:I\rightarrow X$ by
\begin{align*} \gamma(t) = \int_0^t\alpha(s)b(\eta(s))ds, \ \ t \in I.
\end{align*}
We claim that the pair $(\gamma,\eta)$ is a Legendre curve with curvature $(\alpha,\kappa)$. To verify this, we derivate $\gamma$ to obtain $\gamma'(t) = \alpha(t)b(\eta(t))$. Notice that $\eta(t) \dashv_B \gamma'(t)$. Therefore, in view of the previous notation, we indeed have $\xi(t) = b(\eta(t))$, and consequently equality (\ref{1stcurvature}) holds. Now, differentiating $\eta$ yields
\begin{align*} \eta'(t) = u'(t)\frac{d\varphi}{du}(u(t)) = \kappa(t)\xi(t),
\end{align*}
since $\varphi(u)$ is an arc-length parametrization of the unit circle.

\end{proof}

Of course, the next natural question is whether or not such a Legendre curve is uniquely determined if we fix initial conditions $\gamma(0) \in X$ and $\eta(0) \in S$. We give now a positive answer to this question using the standard theory of ordinary differential equations.

\begin{teo}[Uniqueness theorem] Let $(\alpha,\kappa):I\rightarrow\mathbb{R}^2$ be a smooth function and fix $(p,v) \in X\times S$. Then there exists a unique Legendre curve $(\gamma,\eta):I\rightarrow X\times S$ whose curvature is $(\alpha,\kappa)$ and such that $\gamma(0) = p$ and $\eta(0) = v$.
\end{teo}
\begin{proof} From the construction in the previous theorem, it is clear that to determine a vector field $\eta:I\rightarrow S$ such that $\eta'(t) = \kappa(t)b(\eta(t))$ with initial condition $\eta(0) = v$ is equivalent to finding a function $u:I\rightarrow\mathbb{R}$ that solves the initial value problem
\begin{align*} \left\{\begin{array}{ll} u'(t) = \kappa(t), \ \ t \in I \\ u(0) = u_0 \end{array}\right.,
\end{align*}
where $u_0 \in \mathbb{R}$ is such that $v = \varphi(u_0)$. Uniqueness of such a function is guaranteed by the standard theory of ordinary differential equations (see, for instance, \cite{coddington}). \\

Now the tangent vector field $\gamma'(t) = \alpha(t)\xi(t)$ is completely determined (where $\xi(t) = b(\eta(t))$, as usual). Since it is clear that smooth curves with the same tangent vector field must be equal up to translation, the proof is complete.

\end{proof}

As a consequence of the uniqueness theorem, we have a characterization of the Minkowski circle. See \cite[Proposition 2.12]{Fu-Ta3} for the Euclidean version of this characterization.

\begin{prop} A Legendre curve $(\gamma,\eta):I\rightarrow X\times S$ is contained in a Minkowski circle if and only if there exists a constant $c \in \mathbb{R}$ such that $\alpha(t) = c\kappa(t)$ for all $t \in I$.
\end{prop}
\begin{proof} If $\gamma$ is contained in a Minkowski circle of radius $c$, then the circular curvature equals $1/c$ (cf. \cite[Theorem 6.1]{Ba-Ma-Sho}). Therefore, from Lemma \ref{relcurv} it follows that $\alpha(t) = c\kappa(t)$ for every $t \in I$. The converse follows immediately from the uniqueness theorem.

\end{proof}

Let $(\gamma,\eta):I\rightarrow X\times S$ be a Legendre curve with curvature $(\alpha,\kappa)$, and let $T:X\rightarrow X$ be an \emph{isometry} of the plane, i.e., a norm-preserving map. An isometry is called \emph{orientation preserving} if the sign of the fixed determinant form remains invariant under its action. Since Birkhoff orthogonality is defined in terms of distances, it is clear that $(T\gamma,T\eta)$ is still a Legendre curve, and hence it has a curvature function $(\kappa_T,\alpha_T)$.

\begin{teo}[Invariance under isometries] The curvature of a Legendre curve is invariant under an orientation preserving isometry of the plane.
\end{teo}

\begin{proof} In the same notation as above, we have to prove that $\kappa = \kappa_T$ and $\alpha = \alpha_T$. Recall that an isometry of a normed plane must be linear up to translation, and then we may consider it as linear, for the sake of simplicity (cf. \cite{Ba-Ma-Sho}). Hence, from (\ref{1stcurvature}) and (\ref{2ndcurvature}) we have the equalities
\begin{align*}(T\gamma)'(t) = T\gamma'(t) = \alpha(t)T\xi(t) \ \ \mathrm{and} \\ (T\eta)'(t) = T\eta'(t) = \kappa(t)T\xi(t).
\end{align*}
Therefore, in order to prove that $\kappa = \kappa_T$ and $\alpha = \alpha_T$ it suffices to show that $T\xi(t) = b(T\eta(t))$, where we recall that $\xi(t) = b(\eta(t))$. But this comes immediately, since $T\eta(t) \dashv_B Tb(\eta(t))$, $||Tb(\eta(t))|| = ||b(\eta(t))|| = 1$, and $T$ is orientation preserving.

\end{proof}

\begin{remark} Clearly, if the considered isometry is orientation reversing, then we have $\kappa_T = -\kappa$ and $\alpha_T = -\alpha$.
\end{remark}

\section{Ordinary cusps of closed fronts}

A singularity $t_0 \in I$ of a smooth curve $\gamma:I\rightarrow X$ is said to be an \emph{ordinary cusp} if $\gamma''(t_0)$ and $\gamma'''(t_0)$ are linearly independent vectors. Our next statement shows that we can describe an ordinary cusp of a front in terms of the curvature functions of an associated Legendre immersion.

\begin{lemma}\label{ordcusp} Let $\gamma:I\rightarrow X$ be a front, and let $\eta:I\rightarrow S$ be a smooth vector field such that $(\gamma,\eta)$ is a Legendre immersion. A point $t_0 \in I$ is an ordinary cusp if and only if $\alpha'(t_0) \neq 0$, where $\alpha:I\rightarrow\mathbb{R}$ is defined as in \emph{(}\ref{1stcurvature}\emph{)}.
\end{lemma}
\begin{proof} This comes from the two-fold straightforward differentiation of equality (\ref{1stcurvature}). In a singular point $t_0 \in I$ we have
\begin{align*} \gamma''(t_0) = \alpha'(t_0)\xi(t_0) \ \ \mathrm{and} \\ \gamma'''(t_0) = 2\alpha'(t_0)\xi'(t_0) + \alpha''(t_0)\xi(t_0).
\end{align*}
Since $(\gamma,\eta)$ is an immersion, we have that $\eta'(t_0) \neq 0$. Therefore, $\xi'(t_0) = Db_{\eta(t_0)}(\eta'(t_0)) \neq 0$ and $\xi(t_0) \dashv_B \xi'(t_0)$. The desired follows.

\end{proof}

It is clear that the definition of an ordinary cusp does not involve any metric or orthogonality concept fixed in the plane. Indeed, one just needs differentiation of a curve to define an ordinary cusp. The previous lemma, despite being easy and intuitive, shows us that, using the curvature defined by the Minkowski metric, one can characterize an ordinary cusp of a Legendre immersion in the same way as we would do it in the standard Euclidean metric. \\

Going a little further in this direction, we head now to formalize the idea that \emph{the orientation changes when we pass through an ordinary cusp}, and we will do this by using only the machinery defined here. Indeed, since $\mathrm{sgn}[\eta(t),\gamma'(t)] = \mathrm{sgn}(\alpha(t))$, it follows that the orientation of the basis $\{\eta(t),\gamma'(t)\}$ changes. By the last lemma, the sign of $\alpha$ changes at a point $t_0 \in I$ if and only if $\gamma(t_0)$ is an ordinary cusp, and then it follows that the orientation of the basis $\{\eta(t),\gamma'(t)\}$ (well defined in a punctured neighborhood of $t_0$) changes when, and only when, we pass through an ordinary cusp. As a consequence we re-obtain the following well known result.

\begin{prop}\label{propcusp} Let $\gamma:S^1\rightarrow X$ be a closed front, where $S^1$ is the usual circle $\mathbb{R}/\mathbb{Z}$. Then $\gamma$ has an even number of ordinary cusps.
\end{prop}
\begin{proof} The intuitive idea here is that we must pass through an even number of ordinary cusps so that the sign of $[\eta(t),\gamma'(t)]$ is not inverted when we return to the initial point (see Figure \ref{oddcusps}). We will formalize this. \\

Let $\eta:S^1\rightarrow S$ be a normal vector field such that $(\gamma,\eta)$ is a Legendre immersion. We identify $S^1$ with the interval $[0,1]$ and, up to a translation in the parameter, assume that $\gamma(0) = \gamma(1)$ is a regular point. Let $\{t_1<t_2 < \dots < t_m\}$ be the set of all ordinary cusps of $\gamma$, and assume that $m$ is odd. It is clear that the sign of $[\eta(t),\gamma'(t)]$ is constant in each interval $(t_{j-1},t_j)$, and also before $t_1$ and after $t_m$. Since there is no ordinary cusp in the interval $(t_m,1+t_1)$, it also follows that $m$ is even. Otherwise, the sign of $[\eta(t),\gamma'(t)]$ would be distinct in $(t_m,1]$ and $[0,t_1)$.

\end{proof}

\begin{figure}[h]
\centering
\includegraphics{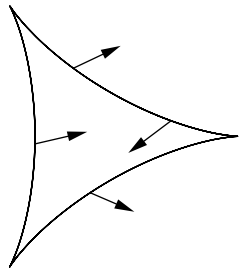}
\caption{A front must have an even number of ordinary cusps.}
\label{oddcusps}
\end{figure}

\begin{remark} In view of Lemma \ref{ordcusp}, an ordinary cusp is a zero-crossing of $\alpha$, and the converse is also true. Since $\gamma$ is closed, we have that $\alpha$ is a periodic smooth function, and then we must have an even number of zero-crossings. This (a little less geometric) argument also works for proving Proposition \ref{propcusp}.
\end{remark}

Our next task is to obtain the \emph{Maslov index} (or \emph{zigzag number}) of a closed front using the generalized curvature of a Legendre immersion (for the Euclidean case this was done in \cite{Fu-Ta}). By \cite{umehara} we are inspired to formulate the following

\begin{definition} Let $t_0 \in I$ be an ordinary cusp of a front $\gamma:I\rightarrow X$ with associated normal field $\eta$. Then, if $[\eta(t_0),\eta'(t_0)] > 0$, we say that $t_0$ is a \emph{zig}, and if $[\eta(t_0),\eta'(t_0)] < 0$, we say that $t_0$ is a \emph{zag}.
\end{definition}

Notice that we always have $[\eta(t_0),\eta'(t_0)] \neq 0$ on an ordinary cusp ($\gamma$ is a front). Geometrically, by this definition we can distinguish whether the normal field rotates counterclockwise or clockwise in the neighborhood of an ordinary cusp, and this is equivalent to the definition given in \cite{umehara}. As one may expect, whether an ordinary cusp is a zig or a zag does not depend on the metric (and consequently not on the orthogonality relation) fixed in the plane. 

\begin{prop} Let $\gamma:I\rightarrow X$ be a front, and let $t_0 \in I$ be an ordinary cusp. Then we have one of the following statements.\\

\noindent\emph{\textbf{(a)}} For every $\varepsilon > 0$ there exist $t_1,t_2 \in (t_0-\varepsilon,t_0+\varepsilon)$ such that $t_1<t_0<t_2$ and $[\gamma'(t_1),\gamma'(t_2)] < 0$, or \\

\noindent\emph{\textbf{(b)}}  for every $\varepsilon > 0$ there exist $t_1,t_2 \in (t_0-\varepsilon,t_0+\varepsilon)$ such that $t_1<t_0<t_2$ and $[\gamma'(t_1),\gamma'(t_2)] > 0$. \\

\noindent In the first case, the cusp is a zig. In the second one, we have a zag.
\end{prop}
\begin{proof} Assume that $(\gamma,\eta)$ is a Legendre immersion, and let $\alpha$ be as in (\ref{1stcurvature}). Since $t_0$ is an ordinary cusp, it follows that for small $\varepsilon > 0$ we have that $\alpha$ has constant and distinct signs in each of the lateral neighborhoods $(t_0-\varepsilon,t_0)$ and $(t_0,t_0+\varepsilon)$ of $t_0$. Therefore, for any $t_1,t_2 \in (t_0-\varepsilon,t_0+\varepsilon)$ with $t_1 < t_0 < t_2$ it holds that $\alpha(t_1)\alpha(t_2) < 0$. Now we write
\begin{align*} [\gamma'(t_1),\gamma'(t_2)] = \alpha(t_1)\alpha(t_2)[\xi(t_1),\xi(t_2)].
\end{align*}
Hence, in $(t_0-\varepsilon,t_0+\varepsilon)$ the sign of $[\gamma'(t_1),\gamma'(t_2)]$ for $t_1 < t_0 < t_2$ depends only on the sign of $[\xi(t_1),\xi(t_2)]$. On the other hand, since $(\gamma,\eta)$ is an immersion, we have that $\eta'(t_0) \neq 0$. Then, taking a smaller $\varepsilon > 0$ if necessary, we may assume that $\eta$ is injective when restricted to the interval $(t_0-\varepsilon,t_0+\varepsilon)$. Consequently, $\xi$ is also injective in $(t_0-\varepsilon,t_0+\varepsilon)$, and the sign of $[\xi(t_1),\xi(t_2)]$ for $t_1,t_2 \in (t_0-\varepsilon,t_0+\varepsilon)$ with $t_1<t_2$ only depends on how $\eta$ walks through the unit circle in this interval (clockwise or counterclockwise).

\end{proof}

A zero-crossing of the curvature function $\kappa$ of a Legendre immersion is called an \emph{inflection point}. One can have a better understanding of the classification of ordinary cusps by no\-ti\-cing that two consecutive ordinary cusps of a frontal have different types if and only if there is an odd number of inflection points between them. Indeed, this follows from the fact that $\mathrm{sgn}[\eta(t),\eta'(t)] = \mathrm{sgn}(\kappa(t)[\eta(t),b(\eta(t))]) = \mathrm{sgn}(\kappa(t))$ and from the continuity of $\kappa$. \\

Let $\gamma:S^1 \rightarrow X$ be a closed front, and let $C_{\gamma}:= \{t_1,...,t_m\}$ be the set of its (ordered) ordinary cusps. Attribute the letter $a$ to a zig, and $b$ to a zag, and form the word $w_{\gamma}:=t_1t_2...t_m$. Since $m$ is even, it follows that the identification of $w_{\gamma}$ in the free product $\mathbb{Z}_2*\mathbb{Z}_2$ (considering the reduction $a^2=b^2=1$) must be of the form $(ab)^k$ of $(ba)^k$. The number $k$ is called the \emph{Maslov index} (or \emph{zigzag number}) of $\gamma$, and it will be denoted by $z(\gamma)$. \\

We will follow \cite{umehara} to obtain the Maslov index in terms of the curvature pair of a Legendre immersion, but now the considered curvature pair is given by Birkhoff orthogonality instead of Euclidean orthogonality. First, let $\mathrm{P}^1(\mathbb{R})$ be the real projective line, and let $[ x : y ]$, defined as $y/x$, be coordinates on it. The curvature pair of a Legendre immersion can then be regarded as the smooth map $k_{\gamma}:S^1\rightarrow \mathrm{P}^1(\mathbb{R})$ given by
\begin{align*} k_{\gamma}(t) = [\alpha(t):\kappa(t)], 
\end{align*}
where $\alpha$ and $\kappa$ are, as usual, given as in (\ref{1stcurvature}) and (\ref{2ndcurvature}). If we identify canonically the projective line with the one-dimensional circle (see Figure \ref{projectiveline}), then we can naturally define the \emph{rotation number} of $k_{\gamma}$ as its absolute number of (complete) turns over the circle, counted with sign depending on the orientation. In the following, we say that a front is \emph{generic} if all of its singular points are ordinary cusps and all of its self-intersections are \emph{double points}, which means that if $t_0 \neq t_1$ and $\gamma(t_0) = \gamma(t_1)$, then $\eta(t_0)$ and $\eta(t_1)$ are linearly independent vectors.

\begin{figure}[h]
\centering
\includegraphics{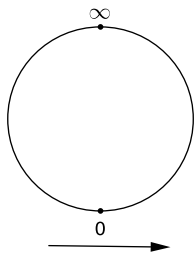}
\caption{Identification $\mathrm{P}^1(\mathbb{R}) \simeq S^1$, with orientation (positive ratios are on the right-hand side).}
\label{projectiveline}
\end{figure}

\begin{teo}\label{teozigzag} Let $\gamma:S^1\rightarrow X$ be a generic closed front. Then the zigzag number of $\gamma$ equals the rotation number of $k_{\gamma}$.
\end{teo}
\begin{proof} Following \cite{umehara}, the strategy of the proof is to count the number of times that $k_{\gamma}$ passes through the point $[0:1]$ (=$\infty$) two consecutive times with the same orientation. First, notice that $k_{\gamma}(t) = [0:1]$ if and only if $t$ is an ordinary cusp. Now observe that the sign of $\kappa(t)/\alpha(t)$ in a punctured neighborhood of a cusp is the same as the sign of $\alpha\kappa$. Therefore, we can decide whether we have a clockwise or a counterclockwise $\infty$-crossing at a singularity $t_0 \in S^1$ looking to the sign of $(\alpha\kappa)'(t_0) = \alpha'(t_0)\kappa(t_0)$. Namely, if $\alpha'(t_0)\kappa(t_0) < 0$, the $\infty$-crossing is counterclockwise, and if $\alpha'(t_0)\kappa(t_0) > 0$, then it is clockwise. \\

Now let $t_0,t_1 \in S^1$ be two consecutive singularities. It is clear that $\alpha'(t_0)$ and $\alpha'(t_1)$ have opposite signs. Therefore, if we have two consecutive zigs, or two consecutive zags, then the associated consecutive $\infty$-crossings have the same orientation, and consequently play no role in the rotation number. On the other hand, a zig followed by a zag (or vice-versa) yields a complete (positive of negative) turn over $\mathrm{P}^1(\mathbb{R})$. This shows what we had to prove.

\end{proof}

\begin{remark} Since the choice of a normal field to turn a front into a Legendre immersion is not unique, as well as the associated curvature pair is not invariant under a re-parametrization of the front, a comment is due. The classification of \emph{all} ordinary cusps will change if we replace the field $\eta$ by $-\eta$, and hence the Maslov index remains the same. Also, it is easily seen that a re-parametrization of the front yields a new curvature pair where both previous curvature functions are multiplied by a same function. Therefore, the map $k_{\gamma}:S^1\rightarrow\mathrm{P}^1(\mathbb{R})$ defined previously is invariant under a re-parametrization of the front, and so is its rotation number.
\end{remark}

We shall describe now another way to obtain the Maslov index of a closed front. In view of Theorem \ref{teozigzag}, what changes is that we count the rotation number of $k_{\gamma}$ by regarding zero-crossings instead of $\infty$-crossings. Geometrically, instead of using types of singular points, we classify inflection points (recall that, for us, an inflection point is a zero-crossing of the curvature function $\kappa$). As the reader will notice, this approach has the advantage of avoiding the use of reductions in free products. \\

Let, as usual, $(\gamma,\eta):S^1 \rightarrow X\times S$ be a generic closed front with associated curvature pair $(\alpha,\kappa)$. We say that an inflection point $t_0 \in S^1$ is a \emph{flip} if $t_0$ is a zero-crossing from negative to positive of $\alpha(t)\kappa(t)$, and a \emph{flop} if $t_0$ is a zero-crossing from positive to negative of $\alpha(t)\kappa(t)$. Notice that every inflection point is a flip or a flop, since $\alpha(t_0)$ does not vanish ($\gamma$ is a front) and $t_0$ is a zero-crossing of $\kappa$. We have

\begin{teo} The Maslov index of a generic closed front $\gamma$ is half the absolute value of the difference between its numbers of flips and flops. In other words,
\begin{align*} z(\gamma) = \frac{1}{2}|\#\mathrm{flip} - \#\mathrm{flop}|, 
\end{align*}
where $\#\mathrm{flip}$ and $\#\mathrm{flop}$ denotes the numbers of inflection points for each respective type. 
\end{teo}
\begin{proof} Before proving the theorem, it is interesting to capture the combinatorial flavor of the problem. Notice first that a flip corresponds to a counterclockwise zero-crossing of $k_{\gamma}$ in $\mathrm{P}^1(\mathbb{R})$, and a flop corresponds to a clockwise zero-crossing. Between two consecutive singularities, we have two possibilities:\\

\noindent (1) The number of inflection points is even. In this case we have the same number of flips and flops, since $\alpha$ does not change its sign between two consecutive zeros. \\

\noindent (2) The number of inflection points is odd. In this case we have $|\#\mathrm{flip}-\#\mathrm{flop}| = 1$ between these singularities. \\

Moreover, successive zero-crossings of $\kappa$ are always alternate, and then we have two con\-se\-cu\-ti\-ve flips (or flops) when there is a singular point between two consecutive inflection points. The reader is invited to draw some concrete examples, to capture the ideas in a better way. We will give an analytic proof, however. As noticed in \cite{Fu-Ta}, the zigzag number is half the absolute value of the degree of the map $k_{\gamma}:S^1\rightarrow \mathrm{P}^1(\mathbb{R})$. Since $S^1$ is path connected, we can calculate the degree of $k_{\gamma}$ by counting the points of the set $k_{\gamma}^{-1}([1:0])$ where the derivative is orientation preserving/reversing. In other words, the degree of $k_{\gamma}$ is the difference between the numbers of counterclockwise and clockwise zero-crossings in $\mathrm{P}^{1}(\mathbb{R})$. Since each counterclockwise zero-crossing corresponds to a flip, and each clockwise zero-crossing corresponds to a flop, we have indeed
\begin{align*} z(\gamma) = \frac{1}{2}\mathrm{deg}(k_{\gamma})= \frac{1}{2}|\#\mathrm{flip} - \#\mathrm{flop}|,
\end{align*}
as we aimed to prove.

\end{proof}

\section{Evolutes and involutes of fronts}\label{evoinvo}

Let $\gamma:I\rightarrow X$ be a smooth regular curve whose circular curvature $k$ does not vanish. Then the \emph{evolute} of $\gamma$ is the curve $e_{\gamma}:I\rightarrow X$ defined as
\begin{align*} e_{\gamma}(t) = \gamma(t) - \rho(t)\eta(t),
\end{align*}
where $\rho(t) := k(t)^{-1}$ is the curvature radius of $\gamma$ at $t \in I$ and $\eta(t)$ is the left normal vector to $\gamma$ at $t \in I$ (both defined as in our Introduction). A \emph{parallel} of $\gamma$ is a curve of the type
\begin{align}\label{parallel} \gamma_d(t) = \gamma(t) + d\eta(t),
\end{align}
for some fixed $d \in \mathbb{R}$. As in the Euclidean case, the singular points of the parallels of $\gamma$ sweep out the evolute of $\gamma$ (see \cite[Section 9]{Ba-Ma-Sho}). Based on this characterization, we will follow \cite{Fu-Ta2} to define the evolute of a front in a Minkowski plane.\\

First, let $(\gamma,\eta):I\rightarrow X$ be a Legendre immersion. Then, using the normal field $\eta$ we can define a \emph{parallel of the front} $\gamma$ exactly by (\ref{parallel}). 

\begin{lemma} A parallel of a front $\gamma:I\rightarrow X$ is also a front. 
\end{lemma}
\begin{proof} Let $(\gamma,\eta)$ be a Legendre immersion. We shall see that $(\gamma_d,\eta)$ is a Legendre immersion. From (\ref{1stcurvature}) and (\ref{2ndcurvature}) we have $\gamma_d'(t) = \gamma'(t) + d\eta'(t) = (\alpha(t)+d\kappa(t))\xi(t)$. Therefore, $\eta(t) \dashv_B \gamma'_d(t)$ for each $t \in I$. It only remains to prove that $(\gamma_d,\eta)$ is an immersion. For this sake, just write down the equations
\begin{align}\label{parallelcurvature} \gamma'_d(t) = (\alpha(t)+d\kappa(t))\xi(t) \ \ \mathrm{and} \ \  \eta'(t) = \kappa(t)\xi(t),
\end{align}
and observe that $\gamma'_d$ and $\eta'$ vanish simultaneously if and only if $\alpha$ and $\kappa$ vanish simultaneously. But this would contradict the hypothesis that $(\gamma,\eta)$ is a Legendre immersion.  

\end{proof}

From now on we will \textbf{always} assume that $\gamma:I\rightarrow X$ is a front, and that the pair $(\gamma,\eta)$ is an associated Legendre immersion whose curvature pair $(\alpha,\kappa)$ is such that $\kappa$ does not vanish. Then, we define the \emph{evolute} of $\gamma$ to be
\begin{align}\label{evolute} e_{\gamma}(t) = \gamma(t) - \frac{\alpha(t)}{\kappa(t)}\eta(t), \ \ t \in I.
\end{align} 

Notice that this definition makes sense (as an extension of the usual evolute of a regular curve) in view of Lemma \ref{relcurv}. Also, observe that a front and its evolute intersect in (and only in) singular points of the front. Further in this direction, and as we have mentioned, the evolute of a front is the set of singular points of the parallels of this front. Now we will prove this.

\begin{prop}\label{propparallel} The set of points of the evolute of a front $\gamma$ is precisely the set of singular points of the parallels of $\gamma$. 
\end{prop}
\begin{proof} For each $t \in I$, the point $e_{\gamma}(t)$ belongs to the parallel given by $d = -\frac{\alpha(t)}{\kappa(t)}$. Since $\gamma'_d(t) = (\alpha(t)+d\kappa(t))\xi(t)$, it follows that $\gamma_d$ is singular at that point. On the other hand, a singular point of a parallel $\gamma_d$ must be given by some $t \in I$ such that $d = -\frac{\alpha(t)}{\kappa(t)}$. 

\end{proof}

We will verify that the evolute of a front is also a front, whose curvature can be obtained in terms of $(\alpha,\kappa)$. To do so, from now on we consider the map $b$ defined only for unit vectors. We do this because the restriction $b|_S:S\rightarrow S$ is bijective, and hence invertible. Let $(\gamma,\eta)$ be a Legendre immersion with associated curvature function given by $(\alpha,\kappa)$, and let $e_{\gamma}$ be its evolute. We will write $\nu(t) = -b^{-1}(\eta(t))$. With these notations, we have

\begin{teo} \label{teoevo}The pair $(e_{\gamma},\nu)$ is a Legendre immersion with associated curvature given by the equalities
\begin{align*} e_{\gamma}'(t) = -\frac{d}{dt}\left(\frac{\alpha(t)}{\kappa(t)}\right)\eta(t) \ \ and \\ \nu'(t) = \beta(t)\eta(t).
\end{align*}
Since $\eta(t) = -b(\nu(t))$, it follows that the curvature pair of $(e_{\gamma},\nu)$ is given by $\left(\frac{d}{dt}\left(\frac{\alpha}{\kappa}\right),-\beta\right)$. Moreover, the function $-\beta(t)$ is given by
\begin{align*} -\beta(t) = \frac{\kappa(t)}{\rho(t)},
\end{align*}
where $\rho(t):=\rho(\nu(t))$ is as defined in Remark \ref{diffb}.
\end{teo}
\begin{proof} The first equation comes straightforwardly by differentiating (\ref{evolute}). For the second, first notice that since $\nu$ is a unit vector field, it follows that $\nu(t) \dashv_B \nu'(t)$, and therefore $\nu'(t)$ is parallel to $\eta(t)$. Notice that already this characterizes the pair $(e_{\gamma},\nu)$ as a Legendre curve. Before showing that this is indeed an immersion, we will prove the expression for $-\beta$. Differentiating $\nu$ yields\\
\begin{align*} \nu'(t) = -Db^{-1}_{\eta(t)}(\eta'(t)) = -\kappa(t)Db^{-1}_{\eta(t)}(b(\eta(t))) = -\kappa(t)\overline{\rho}(t)\eta(t),
\end{align*}
for some function $\overline{\rho}$. We will prove then that $\overline{\rho}(t) = \rho(\nu(t))^{-1}$. For this sake, let $v = \eta(t)$ and write down the equalities
\begin{align*} Db_{b^{-1}(v)}(v) = \rho(b^{-1}(v))b(v) \ \ \mathrm{and} \\ Db^{-1}_v(b(v)) = \overline{\rho}(v)v.
\end{align*}
Therefore,
\begin{align*}\overline{\rho}(v)v = Db^{-1}_v(b(v)) = \frac{1}{\rho(b^{-1}(v))}Db^{-1}_v(Db_{b^{-1}(v)}(v)) = \frac{1}{\rho(b^{-1}(v))}v,
\end{align*}
and since $\rho(b^{-1}(v)) = \rho(-\nu(t)) = \rho(\nu(t))$, the desired follows. Now, since by our hypothesis the function $\kappa$ does not vanish, it follows that $(e_{\gamma},\nu)$ is, in fact, a Legendre immersion. 

\end{proof}

\begin{remark} Unlike in the Euclidean subcase, here the function $\rho(t)$ appears. This function carries, somehow, the ``distortion" of the unit circle of the considered norm with respect to the Euclidean unit circle. Such function appears even in Radon planes.
\end{remark}

An evolute of a regular curve in the Euclidean plane is the envelope of its normal lines, and the same holds for the evolute of a regular curve in a Minkowski plane, of course when we replace inner product orthogonality by Birkhoff orthogonality (see \cite{craizer}). From (\ref{evolute}) and Theorem \ref{teoevo} it follows that the tangent line of the evolute $e_{\gamma}$ of a Legendre immersion $(\gamma,\eta)$ at $t \in I$ is precisely the normal line of $\gamma$ at $\gamma(t)$. Therefore, the evolute of a Legendre immersion can be regarded as the envelope of the normal line field of the immersion. \\

The family of normal lines of a Legendre immersion $(\gamma,\eta)$ is the zero set of the function $F:I\times X \rightarrow \mathbb{R}$ given by $F(t,v) = [\gamma(t) - v, \eta(t)]$. Indeed, for each fixed $t \in I$ the zero set of $F_t(v):=F(t,v)$ is the normal line of $\gamma$ at $\gamma(t)$. Therefore, we could expect that the points, for which both $F$ and its derivative with respect to $t$ vanish, describe the evolute of the Legendre immersion (see \cite{Bru-Gi}). This is indeed true, as we shall see next.

\begin{prop}\label{envevo} For the function $F:I\times X\rightarrow \mathbb{R}$ defined above we have that 
\begin{align*}F(t,v) = \frac{\partial F}{\partial t}(t,v) = 0 \  \mathit{if \ and \ only \ if} \  \ v = \gamma(t)- \frac{\alpha(t)}{\kappa(t)}\eta(t). 
\end{align*}
Therefore, the envelope of the normal line field of a Legendre immersion is precisely its evolute. 
\end{prop}
\begin{proof} It is clear that $F(t,v) = 0$ if and only if $\gamma(t)- v = \lambda\eta(t)$ for some $\lambda \in \mathbb{R}$. Differentiating, we have
\begin{align*} \frac{\partial F}{\partial t}(t,v) = [\gamma'(t),\eta(t)] + [\gamma(t) - v,\eta'(t)] = \alpha(t)[\xi(t),\eta(t)] + \kappa(t)[\gamma(t)-v,\xi(t)].
\end{align*} 
Hence, $F(t,v) = \frac{\partial F}{\partial t}(t,v) = 0$ if and only if $\gamma(t) - v = \lambda \eta(t)$ and $(\alpha(t) - \lambda\kappa(t))[\xi(t),\eta(t)] = 0$. Since $[\xi(t),\eta(t)]$ does not vanish, the desired follows. 

\end{proof}

The involute of a regular curve $\gamma$ is a curve whose evolute is $\gamma$ (see \cite{Ba-Ma-Sho} and \cite{craizer}). We can easily extend this definition to our new context, in a way similar to as it was done for the Euclidean subcase in\cite{Fu-Ta4}. An \emph{involute} of a Legendre immersion $(\gamma,\eta):[0,c]\rightarrow X\times S$ whose curvature $\kappa$ does not vanish is a Legendre immersion whose evolute is $(\gamma,\eta)$. 

\begin{teo} Let $(\gamma,\eta):[0,c]\rightarrow X\times S$ be a Legendre immersion with curvature pair $(\alpha,\kappa)$, and assume that $\kappa$ does not vanish. For any $d \in \mathbb{R}$, the map $(\sigma,\xi):[0,c]\rightarrow X\times S$, where $\xi(t) = b(\eta(t))$, as usual, and
\begin{align}\label{involute}\sigma(t) = \gamma(0) - \int_0^t\left(\int_0^s\alpha(\tau)d\tau\right)\xi'(s)ds + d\xi(t)
\end{align}
is a Legendre immersion with curvature pair 
\begin{align}\label{involutecurv} \left(\kappa(t)\rho(\eta(t))\left(-d+\int_0^t\alpha(\tau)d\tau\right),-\kappa(t)\rho(\eta(t))\right):[0,c]\rightarrow \mathbb{R}^2,
\end{align}
and with $\rho$ defined as in \emph{(}\ref{xideriv}\emph{)}. Moreover, $(\sigma,\xi)$ is an involute of $(\gamma,\eta)$. 
\end{teo}
\begin{proof} First, differentiation of $\sigma$ yields
\begin{align*} \sigma'(t) = \left(d-\int_0^t\alpha(\tau)d\tau\right)\xi'(t) = \kappa(t)\rho(\eta(t))\left(-d+\int_0^t\alpha(\tau)d\tau\right)b(\xi(t)),
\end{align*}
where the last equality comes from (\ref{xideriv}). Notice that $\xi(t) \dashv \sigma'(t)$ for each $t \in [0,c]$, and hence the pair $(\sigma,\xi)$ is a Legendre curve. The derivative of the normal field $\xi$ is given by the equality (\ref{xideriv}), and then the curvature pair of $(\sigma,\xi)$ is precisely the one given in (\ref{involutecurv}). Since $\kappa(t)\rho(\eta(t))$ does not vanish, it follows that $(\sigma,\xi)$ is indeed an immersion. \\

It remains to show that the evolute of $(\sigma,\xi)$ is $(\gamma,\eta)$. From the definition, the evolute of $\sigma$ is the curve
\begin{align*} e_{\sigma}(t) = \sigma(t) + \left(-d+\int_0^t\alpha(\tau)d\tau\right)\xi(t), \ \ t \in [0,c].
\end{align*}
Notice that $e_{\sigma}(0) = \gamma(0)$. Therefore, it suffices to show that $e_{\sigma}$ and $\gamma$ have the same derivative. A simple calculation gives $e_{\sigma}'(t) = \alpha(t)\xi(t) = \gamma'(t)$, and this concludes the proof. 

\end{proof}

Observe that a front has a family of involutes (with parameter $d \in \mathbb{R}$), which is a family of parallel curves. In view of Proposition \ref{propparallel}, a front can be characterized as the set of singular points of these involutes. This remark is the reason for our slightly different approach in comparison with \cite{Fu-Ta4}. Also one would expect that this happens since any of the parallels has the same normal vector field, and therefore yields the same envelope (which is $\gamma$, in view of Proposition \ref{envevo}).

\section{Singular points and vertices of Legendre immersions}

A point where the derivative of the curvature of a regular curve vanishes is usually called a \emph{vertex}. We shall extend this definition to fronts in normed planes, in the same way as it was done in \cite{Fu-Ta2} for the Euclidean subcase. Let $(\gamma,\eta):[0,c]\rightarrow X\times S$ be a Legendre immersion with curvature pair $(\alpha,\kappa)$, and assume that $\kappa$ does not vanish. We say that $t_0 \in [0,c]$ is a \emph{vertex} of the front $\gamma$ (or of the associated Legendre immersion) if 
\begin{align*} \frac{d}{dt}\left(\frac{\alpha}{\kappa}\right)(t_0) = 0.
\end{align*}
Notice that, as in the regular case, a vertex of a front corresponds to a singular point of its evolute (and that the converse also holds, of course). A vertex which is a regular point of $\gamma$ is said to be a \emph{regular vertex}. As one will suspect, we can re-obtain the vertex in terms of the function $F$ which describes the normal line field of the front (see Proposition \ref{envevo}).

\begin{lemma} Let $(\gamma,\eta):I\rightarrow X\times S$ be a Legendre immersion, and let $F:I\times X\rightarrow\mathbb{R}$ be defined as $F(t,v) = [\gamma(t)-v,\eta(t)]$.    Therefore, $t_0 \in I$ is a vertex of $\gamma$ if and only if 
\begin{align*} \frac{\partial^2F}{\partial t^2}(e_{\gamma}(t_0),t_0) = 0.
\end{align*}

\end{lemma}
\begin{proof} A simple calculation gives
\begin{align*} \frac{\partial^2F}{\partial t^2}(t,v) = [\gamma''(t),\eta(t)]+[\gamma(t)-v,\eta''(t)].
\end{align*}
Hence, in a point $(e_{\gamma}(t),t)$ we have
\begin{align*} \frac{\partial^2F}{\partial t^2}(e_{\gamma}(t),t) = [\gamma''(t),\eta(t)]+\left[\frac{\alpha(t)}{\kappa(t)}\eta(t),\eta''(t)\right] = [\eta(t),\xi(t)]\left(\frac{\alpha(t)}{\kappa(t)}\kappa'(t)-\alpha'(t)\right),
\end{align*}
and it is clear that the latter vanishes at $t_0 \in I$ if and only if $\frac{d}{dt}\left(\frac{\alpha}{\kappa}\right)(t_0) = 0$.

\end{proof}

The easy observation that the function $\alpha/\kappa$ must have a local extremum strictly between two consecutive singular points leads to the following version of the \emph{four vertex theorem}.

\begin{prop} Any of the following conditions is sufficient for a closed front $\gamma:S^1\rightarrow X$ to have at least four vertices: \\

\noindent\textbf{\emph{(a)}} $\gamma$ has at least four singular points,  \\

\noindent\textbf{\emph{(b)}} $\gamma$ has at least two singular points which are not ordinary cusps.
\end{prop}
\begin{proof} For \textbf{(a)}, notice that if $\gamma$ has at least four singular points, then $\alpha/\kappa$ has at least four local extrema, each of them corresponding to a vertex. For \textbf{(b)}, just notice that a singular point which is not an ordinary cusp is, in particular, a vertex. Indeed, the derivative
\begin{align*} \frac{d}{dt}\left(\frac{\alpha}{\kappa}\right)(t) =\frac{\alpha'(t)\kappa(t)-\alpha(t)\kappa'(t)}{\kappa(t)^2} 
\end{align*}
vanishes whenever $\alpha(t) = \alpha'(t) = 0$ (and this happens in a singular point which is not an ordinary cusp, see Lemma \ref{ordcusp}). In addition to these vertices, the existence of two regular vertices (guaranteed by the two singular points) finishes the argument.

\end{proof}

It is easy to see that a singular point of a Legendre curve in a Minkowski plane is still a singular point if we change the considered norm. Moreover, an ordinary cusp remains an ordinary cusp. However, a vertex of a Legendre curve may not be a vertex of it if we change the norm of the plane. Indeed, every point of a circle is a vertex (the circular curvature is constant), this is no longer the case when we change the norm. \\

But somehow we can still relate numbers of singular points to numbers of vertices. Since there is at least one vertex strictly between two consecutive singular points of a front, we have that the number of vertices of a closed front is greater than or equal to its number of singular points. Therefore, if $\Sigma(\gamma)$ and $V(\gamma)$ denote the set of singular points and the set of vertices of $\gamma$, respectively, and $\sigma$ is an involute of $\gamma$, then we have
\begin{align*} \#\Sigma(\sigma) \leq \#\Sigma(\gamma) \leq \#V(\gamma),
\end{align*}
where the first inequality comes from the observation that the vertices of $\sigma$ correspond to the singular points of $\gamma$, since $\gamma$ is the evolute of $\sigma$. This observation is proved for the Euclidean subcase in \cite{Fu-Ta4}, and what we wanted to show is that it only depends on the fact that there always exists at least one vertex between two consecutive singular points of a Legendre curve. \\

\section{Contact between Legendre curves}

The concept of \emph{contact} between regular plane curves intends, intuitively, to describe how ``si\-mi\-lar" two curves are in a neighborhood of a point. In \cite[Section 3]{Fu-Ta} this notion is extended to Legendre curves as follows: given $k \in \mathbb{N}$, two Legendre curves $(\gamma_1,\eta_1):I_1\rightarrow X\times S$ and $(\gamma_2,\eta_2):J\rightarrow X\times S$ are said to have \emph{k-th order contact} at $t = t_0$, $u = u_0$ if 
\begin{align*} \frac{d^j}{dt^j}(\gamma_1,\eta_1)(t_0) = \frac{d^j}{du^j}(\gamma_2,\eta_2)(u_0) \ \ \mathrm{for} \ \ j = 0,...,k-1 \ \ \mathrm{and}\\ \frac{d^k}{dt^k}(\gamma_1,\eta_1)(t_0) \neq \frac{d^k}{du^k}(\gamma_2,\eta_2)(u_0).
\end{align*}
If only the first condition holds, then we say that the curves have \emph{at least k-th order contact} at $t = t_0$ and $u = u_0$. In the mentioned paper, this was defined exactly in the same way, but considering that the normal vector field of each Legendre curve is the one given by the Euclidean orthogonality. We shall see that if two Legendre curves have $k$-th order contact for a given fixed norm, then they have $k$-th order contact for any norm. 

\begin{prop} Let $(\gamma,\eta)$ and $(\overline{\gamma},\overline{\eta})$ be Legendre curves which have $k$-th order contact at $t = t_0$ and $u = u_0$. Therefore, changing the norm of the plane, the new Legendre curves (derived in the same sense as discussed in the last paragraph of Section \ref{curvature}) still have $k$-th order contact at $t = t_0$ and $u = u_0$. 
\end{prop}
\begin{proof} Assume that $||\cdot||_1$ and $||\cdot||_2$ are smooth and strictly convex norms in the plane with unit circles $S_1$ and $S_2$, respectively. Denote by $h:S_1\rightarrow S_2$ the map introduced in the last paragraph of Section \ref{curvature}, which takes each vector $v \in S_1$ to the vector $h(v) \in S_2$ such that $S_2$ is supported at $h(v)$ by the same direction which supports $S_1$ at $v$. \\  

Let $(\gamma,\eta)$ be a Legendre curve in the norm $||\cdot||_1$. Then $(\gamma,h\circ\eta)$ is a Legendre curve in the norm $||\cdot||_2$. Writing $\nu = h\circ\eta$, the strategy is to prove that, for any $m \in \mathbb{N}$, the $m$-th derivative of $\nu$ at $t = t_0$ only depends on $h$ and $\eta^{(j)}(t_0)$ for $j = 0,...,m$. For this sake, we just notice that
\begin{align*} \nu^ {(m)}(t_0) = D^mh_{\left(\eta(t_0),\eta' (t_0),...,\eta^ {(m-1)}(t_0)\right)}\left(\eta^{(m)}(t_0)\right),
\end{align*}
where $D^mh_{\left(\eta(t_0),\eta' (t_0),...\eta^ {(m-1)}(t_0)\right)}$ is the usual $m$-th derivative of the map $h$, which is a linear map defined over $T_{\eta^{(m-1)}(t_0)}\left(...\left(T_{\eta'(t_0)}\left(T_{\eta(t_0)}S_1\right)\right)\right)$. Since this derivative clearly depends only on $h$ and $\eta^{(j)}(t_0)$ for $j = 0,...,m$, it follows that $\nu^{(m)}(t_0)$ also only depends on it. \\

Hence, if a change of the norm carries over the normal fields $\eta$ and $\overline{\eta}$ to $\nu$ and $\overline{\nu}$, respectively, then we have that $\nu^{(j)}(t_0) = \overline{\nu}^{(j)}(u_0)$ for every $j = 0,...,k$ if and only if the same happens for $\eta$ and $\overline{\eta}$ (the ``only if" part comes from $h$ being an immersion for all $v \in S_1$). 

\end{proof}

It is a well known fact that the contact between two regular curves can be characterized by means of their curvatures. In \cite[Theorem 3.1]{Fu-Ta} this is extended to Legendre curves using the developed curvature functions. We shall verify that we can obtain an analogous result (in one of the directions, only) when we are not working in the Euclidean subcase. 

\begin{teo} Let $(\gamma_1,\eta_1):I_1\rightarrow X\times S$ and $(\gamma_2,\eta_2):I_2\rightarrow X\times S$ be Legendre curves with curvature pairs $(\alpha_1,\kappa_1)$ and $(\alpha_2,\kappa_2)$, respectively. If these curves have at least k-th order contact at $t = t_0$ and $u = u_0$, then
\begin{align*} \frac{d^j}{dt^j}(\alpha_1,\kappa_1)(t_0) = \frac{d^j}{du^j}(\alpha_2,\kappa_2)(u_0), \ \ \mathrm{for} \ \ j = 0,...,k-1.
\end{align*}
However, the converse may not be true (even up to isometry) if the norm is not Euclidean.
\end{teo}
\begin{proof} The proof is essentially the same as in the mentioned theorem in \cite{Fu-Ta}. Derivating (\ref{1stcurvature}) and (\ref{2ndcurvature}), we have the equalities
\begin{align*} \gamma^{(k)}(t) = \sum_{j=0}^k\binom{k}{j}\alpha^{(j)}(t)\xi^{(k-j)}(t) \ \ \mathrm{and} \\
\eta^{(k)}(t)=\sum_{j=0}^k\binom{k}{j}\kappa^{(j)}(t)\xi^{(k-j)}(t).
\end{align*}
If $k = 1$, then we have $\alpha_1'(t_0)\xi_1(t_0) = \alpha_2'(u_0)\xi_2(u_0)$ and $\kappa_1(t_0)\xi_1(t_0) = \kappa_2(u_0)\xi_2(u_0)$. Since $\eta_1(t_0) = \eta_2(u_0)$, it follows that $\xi_1(t_0) = \xi_2(u_0)$, and hence we have $\alpha_1(t_0) = \alpha_2(u_0)$ and $\kappa_1(t_0) = \kappa_2(t_0)$. Regarding higher order contact, one just has to procceed inductively by using the previous differentiation formulas. \\

We illustrate the fact that the converse does not necessarily hold if the norm is not Euclidean with a constructive example. Take two disjoint arcs $\gamma_1$ and $\gamma_2$ in the unit circle which do not overlap under an isometry (the existence of such arcs is guaranteed by \cite[Proposition 7.1]{Ba-Ma-Sho}). Assume that these arcs are parametrized by arc-length and choose parameters $t_0$, $u_0$ such that the supporting directions to $\gamma_1(t_0)$ and $\gamma_2(u_0)$ are distinct. These arcs, together with their respective normal vector fields ($\eta_1$ and $\eta_2$, say), are Legendre curves whose curvatures and derivatives of curvatures coincide. However, there is no isometry carrying $\gamma_1(t_0)$ to $\gamma_2(u_0)$ and $\eta_1(t_0)$ to $\eta_2(u_0)$, and hence these curves do not have contact of any order up to isometry. 

\end{proof}

\section{Pedal curves of frontals}

A \emph{pedal curve} of a regular curve $\gamma$ is usually defined to be the locus of the orthogonal projections of a fixed point $p$ to the tangent lines of $\gamma$. The existence of a tangent field allows us to carry over this definition straightforwardly to frontals. 

\begin{definition} Let $(\gamma,\eta):I\rightarrow X\times S$ be a Legendre curve, and let $\xi(t) = b(\eta(t))$, as usual. Fix a point $p \in X$. The \emph{pedal curve} of the frontal $\gamma$ with respect to $p$ is the curve $\gamma_p:I\rightarrow X$ which associates to each $t \in I$ the unique point $\gamma_p(t)$ of the line $s\mapsto \gamma(t)+s\xi(t)$ such that $\gamma_p(t)-p \dashv_B \xi(t)$ (see Figure \ref{pedalfig}). In other words, $\gamma_p(t)$ is the intersection of the parallel to $\eta(t)$ drawn through $p$ with the tangent line of $\gamma$ at $\gamma(t)$. 
\end{definition}

\begin{figure}[h]
\centering
\includegraphics{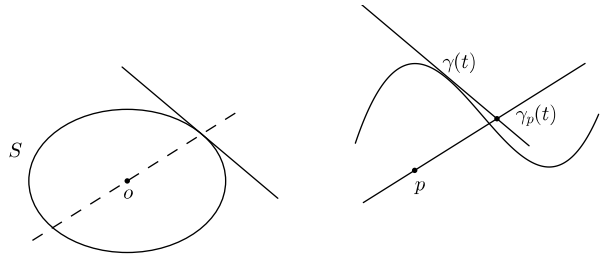}
\caption{Constructing a pedal curve.}
\label{pedalfig}
\end{figure}

It is useful, however, to have a formula for the pedal curve which we can work with. For this sake, fix $t \in I$ and let $\alpha, \beta \in \mathbb{R}$ be constants such that $\gamma_p(t) = \gamma(t)+\alpha\xi(t)$ and $p-\gamma(t) = \beta\eta(t)$. From the vectorial sum $\alpha\xi(t) + \beta\eta(t) = p - \gamma(t)$ we have
\begin{align*} \alpha[\xi(t),\eta(t)] = [p-\gamma(t),\eta(t)].
\end{align*}
Since $\eta(t) \dashv_B \xi(t)$ and the basis $\{\eta(t),\xi(t)\}$ is positively oriented, the above equality reads $\alpha||\xi(t)||_a = [\gamma(t)-p,\eta(t)]$. Hence,
\begin{align}\label{pedal} \gamma_p(t) = \gamma(t) + [\gamma(t) - p,\eta(t)]\xi_a(t),
\end{align}
where $\xi_a(t) = \frac{\xi(t)}{||\xi(t)||_a}$ is the vector $\xi$ normalized in the anti-norm. \\

Notice that from our geometric definition it follows that a pedal curve of a frontal with respect to a given point does not depend on the parametrization of the frontal. We give an illustrated example of a pedal curve of a regular curve.\\

\begin{example}\label{examp} Consider the space $\mathbb{R}^2$ endowed with the usual $l_p$ norm for some $1 < p < +\infty$, and let $q \in \mathbb{R}$ be such that $1/p + 1/q = 1$.  A simple calculation shows that the right pedal curve of the unit circle with respect to the point $(0,1)$ is obtained by joining the curve
\begin{align*} \sigma(t) = \left\{\begin{array}{ll}\left(t^{1/p} - t^{1/p}(1-t)^{1/q},t+(1-t)^{1/p}\right), \ t \in [0,1] \\\left((2-t)^{1/p}+(2-t)^{1/p}(t-1)^{1/q},2 - t - (t-1)^{1/p}\right),\ t \in [1,2] \end{array}\right.
\end{align*}
with its reflection through the $y$-axis. Figure \ref{lpnormrightpedal} illustrates the case $p = 3$.\\
\end{example}

\begin{figure}[h]
\centering
\includegraphics[scale=1.3]{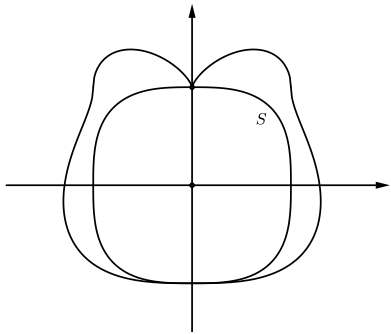}
\caption{The unit circle of $l_3$ and its pedal curve with respect to the point $(0,1)$.}
\label{lpnormrightpedal}
\end{figure}

As an interesting property of pedal curves, we will prove that a frontal can be regarded as the envelope of a certain family of lines defined by (any) one of its pedal curves.

\begin{prop} Let $\gamma_p$ be a pedal curve of a frontal $\gamma$. Then $\gamma$ is the envelope of the family of lines
\begin{align*} \left\{l_t:s \mapsto \gamma_p(t)+sb(\gamma_p(t)-p)\right\}_{t\in I},
\end{align*}
where $l_{t_0}$ is defined by taking limits if $p = \gamma_p(t_0)$ for some $t_0\in I$. Similarly as in Proposition \ref{envevo}, if $F:I\times X\rightarrow \mathbb{R}$ is given by $F(t,v) = [\gamma_p(t)-v,b(\gamma_p(t)-p)]$, then 
\begin{align*} F(t,v) = \frac{\partial F}{\partial t}(t,v) = 0 \ \ \mathit{if \ and \ only \ if} \ v = \gamma(t).
\end{align*}
In particular, any frontal is a pedal curve of some curve in the plane.
\end{prop}
\begin{proof} We may assume, without loss of generality, that locally $b(\gamma_p(t) - p) = \xi(t)$. Derivating $F$ and applying (\ref{pedal}) yields
\begin{align*} \frac{\partial F}{\partial t}(t,v) = [\gamma_p'(t),\xi(t)] + [\gamma_p(t)-v,\xi'(t)] = \\ = [\gamma(t)-p,\eta(t)]\cdot[\xi_a'(t),\xi(t)] + [\gamma_p(t)-v,\xi'(t)].
\end{align*}
 Assume that $F(t,v) = \frac{\partial F}{\partial t}(t,v) =0$. From $F(t,v) = 0$ we have that $\gamma_p(t) - v = \alpha\xi(t)$ for some $\alpha \in \mathbb{R}$. Due to the other equality and to the above calculation, we get
\begin{align*} 0 = [\gamma(t)-p,\eta(t)]\cdot[\xi_a'(t),\xi(t)] + \alpha[\xi(t),\xi'(t)].
\end{align*}
From the definition of $\xi_a(t)$ it follows that
\begin{align*} [\xi_a'(t),\xi(t)] = \frac{[\xi(t),\xi'(t)]}{[\xi(t),\eta(t)]},
\end{align*}
and substituting this in the previous equality yields immediately the equality $\alpha[\xi(t),\eta(t)] = [p-\gamma(t),\eta(t)]$. Therefore, from (\ref{pedal}) we have $v = \gamma(t)$. The converse is straightforward.

\end{proof}

As usual, we assume that $\gamma:I\rightarrow X$ is a frontal with associated normal field $\eta$ whose curvature is $(\alpha,\kappa)$, and we also assume that $\xi$ and $\xi_a$ are defined as before. Notice that differentiation of (\ref{pedal}) yields
\begin{align*} \gamma_p'(t) = \gamma'(t) + [\gamma'(t),\eta(t)]\xi_a(t) + [\gamma(t)-p,\eta'(t)]\xi_a(t) + [\gamma(t) - p,\eta(t)]\xi'_a(t) = \\ = [\gamma(t)-p,\eta'(t)]\xi_a(t) + [\gamma(t)-p,\eta(t)]\xi_a'(t),
\end{align*}
where the second equality is justified since $[\gamma'(t),\eta(t)]\xi_a(t) = -\gamma'(t)$. From the definition of $\xi_a$ and from the equality (\ref{xideriv}), the above equality may be written as
\begin{align}\label{pedalderiv} \gamma_p' = \frac{\kappa}{[\eta,\xi]}\left(\left([\gamma - p,\xi]+\rho\frac{[\gamma-p,\eta]\cdot[\eta,b(\xi)]}{[\eta,\xi]}\right)\xi
- \rho[\gamma-p,\eta]b(\xi)\right),
\end{align}
where we are omitting the variable $t$ for the sake of having a clearer notation. Also we are denoting $\rho=\rho(t) = \rho(\eta(t))$. Notice that if $t_0$ is a point where $\kappa$ vanishes, then $t_0$ is a singular point of the pedal curve $\gamma_p$. Also, if $p$ is a point of $\gamma$, then it is also a singular point of $\gamma_p$. Finally, the only remaining possibility of $\gamma_p$ having a singular point would be if
\begin{align*} [\gamma(t)-p,\xi(t)] + \rho(t)\frac{[\gamma(t)-p,\eta(t)]\cdot[\eta(t),b(\xi(t))]}{[\eta(t),\xi(t)]} = 0 \ \ \mathrm{and} \\ [\gamma(t) - p,\eta(t)] = 0,
\end{align*}
but if the second equality holds, and $p \notin \gamma(I)$, then the first equality does not hold. Indeed, if $\gamma(t) - p \neq 0$, then either $[\gamma(t)-p,\xi(t)] \neq 0$ or $[\gamma(t)-p,\eta(t)] \neq 0$. If $p$ is a point of $\gamma$, then $\gamma_p$ is not necessarily a frontal (a counterexample is given by Example \ref{examp}, in view of Proposition \ref{propcusp}). However, based on this observation we can prove that if $p \notin \gamma(I)$, then $\gamma_p$ is a frontal.

\begin{teo} Let $(\gamma,\eta):I\rightarrow X\times S$ be a Legendre curve with curvature $(\alpha,\kappa)$. If $p \in X\setminus\gamma(I)$, then the pedal curve $\gamma_p$ is a frontal. Moreover, the singular points of $\gamma_p$ correspond exactly to the points where $\kappa$ vanishes.  
\end{teo}
\begin{proof} The equality (\ref{pedalderiv}) can be written as
\begin{align*} \gamma'_p(t) = \frac{\kappa(t)}{[\eta(t),\xi(t)]}\zeta(t),
\end{align*}
where $\zeta(t)$ is the non-vanishing vector field
\begin{align*} \zeta = \left([\gamma-p,\xi]+\rho\frac{[\gamma-p,\eta]\cdot[\eta,b(\xi)]}{[\eta,\xi]}\right)\xi - \rho[\gamma-p,\eta]b(\xi).
\end{align*}
Here again we omitted, for the sake of simplicity, the parameter. Therefore, abusing of the notation and setting $\nu(t) = b^{-1}(\zeta(t))$, we have that $(\gamma_p,\nu)$ is a Legendre curve (with tangent field given by $\zeta/||\zeta||$). Also, since $\zeta$ does not vanish, it follows that $t \in I$ is a singular point of the pedal curve $\gamma_p$ if and only if $\kappa(t) = 0$.

\end{proof}

\bibliography{bibliography.bib}

\end{document}